\documentclass[12pt,dvipsnames]{article}

\usepackage{amsmath}
\usepackage{amsthm}
\usepackage{amsfonts}
\usepackage{stmaryrd}
\usepackage{amssymb}
\usepackage[all]{xy}
\usepackage{comment}
\usepackage{enumitem}

\usepackage{tikz-cd}
\newcommand\N{\mathbb{N}}
\newcommand\R{\mathbb{R}}

\newcommand\Q{\mathbb{Q}}
\newcommand\Z{\mathbb{Z}}

\newcommand{\II}{\mathrm{I\!I}}

\theoremstyle{plain}
\newtheorem{theoremAlph}{Theorem}
\newtheorem{corollaryAlph}[theoremAlph]{Corollary}

\newtheorem{theorem}{Theorem}[section]
\newtheorem{lemma}[theorem]{Lemma}
\newtheorem{corollary}[theorem]{Corollary}
\newtheorem{proposition}[theorem]{Proposition}

\newtheorem{observation}[theorem]{Observation}

\theoremstyle{definition}
\newtheorem{definition}[theorem]{Definition}

\theoremstyle{remark}
\newtheorem{remark}[theorem]{Remark}

\theoremstyle{definition} 
\newtheorem*{ack}{Acknowledgements}

\usepackage{color}

\usepackage{xcolor}

\usepackage{soul}
\setstcolor{ForestGreen}

\usepackage{hyperref}
\hypersetup{
	colorlinks=true,
	linkcolor=blue,
	citecolor=blue,
	urlcolor=blue
}

\advance \topmargin by -\headheight
\advance \topmargin by -\headsep
     
\evensidemargin \oddsidemargin
\makeatletter
\newcommand{\subjclass}[2][1991]{%
	\let\@oldtitle\@title%
	\gdef\@title{\@oldtitle\footnotetext{#1 \emph{Mathematics subject classification.} #2}}%
}

\begin{document}

\renewcommand{\labelenumi}{(\roman{enumi})}
\renewcommand{\labelenumii}{(\alph{enumii})}

\title{A generalization of the Perelman gluing theorem and applications}
\author{Philipp Reiser\thanks{The author acknowledges funding by the SNSF-Project 200020E\textunderscore 193062 and the DFG-Priority programme SPP 2026.} \ and David J. Wraith}
\subjclass[2020]{53C20}

\maketitle

\begin{abstract}
	
We extend a positive Ricci curvature gluing theorem of Perelman to a range of positive intermediate curvature conditions, ranging from positive scalar curvature up to (and including) positive sectional curvature. As an application of this, we demonstrate that the observer moduli space of metrics with positive intermediate Ricci curvatures can have non-trivial higher homotopy groups. Further applications include deriving a sufficient condition for the existence of a metric with positive intermediate Ricci curvature and totally geodesic boundary.
	
\end{abstract}

\section{Introduction}\label{intro}

This paper concerns positive {\it intermediate curvatures}, in the following sense. Recall that a symmetric bilinear map $A\colon V\times V\to \R$ on a Euclidean vector space $(V,\langle\cdot,\cdot\rangle)$ is \emph{$k$-positive}, if for every collection $e^1,\dots,e^k$ of orthonormal vectors the sum $\sum_{i=1}^{k}A(e^i,e^i)$ is positive. 
\begin{definition}
	Let $(M,g)$ be a Riemannian manifold of dimension $n$.
		\begin{enumerate}
			\item Given an integer $1\le k \le n-1$, we say that $(M,g)$ has \emph{positive $k^{th}$-intermediate Ricci curvature}, denoted $Ric_k>0$, if for all $p\in M$ and all collections $v,e^1,\dots,e^k$ of orthonormal vectors in $T_pM$, the sum $\sum_{i=1}^k K(v,e^i)$ is positive, where $K$ denotes the sectional curvature.
			\item Given an integer $1\leq k\le n$, we say that $(M,g)$ has \emph{$k$-positive Ricci curvature}, denoted $scal_k>0$, if for all $p\in M$ the Ricci tensor $Ric$ on $T_p M$ is $k$-positive.
		\end{enumerate}
\end{definition}

The notions of non-negative $k^{th}$-intermediate Ricci curvature and $k$-non-negative Ricci curvature are defined analogously. Note that for an $n$-dimensional manifold, the condition $Ric_k>0$ interpolates between positive sectional curvature (when $k=1$) and positive Ricci curvature (when $k=n-1$), and the condition $scal_k>0$ interpolates between positive Ricci curvature (when $k=1$) and positive scalar curvature (when $k=n$).

Although these notions of curvature have been present in the literature for years (decades in the case of intermediate Ricci curvatures), there has been a dramatic increase in interest in recent times. There is now a rich collection of papers which feature intermediate curvatures. For an up-to-date list of these works see \cite{Mo}.

The question of which manifolds admit a Riemannian metric of $Ric_k>0$ or $scal_k>0$ for some $k$ is, on the whole, wide open as there is no topological obstruction known for the existence of a metric of $Ric_k>0$ (resp.\ $scal_k>0$) with $k>1$ on a closed manifold that does not already hold for $Ric>0$ (resp.\ $scal>0$). Only in the presence of a non-empty boundary it is known that the conditions $Ric_k>0$ have certain topological implications which vary with $k$, see \cite{Wu}. Examples of Riemannian manifolds of $Ric_k>0$ have been constructed in \cite{AQZ} and \cite{DGM} by considering homogeneous spaces, and in \cite{RW1,RW2} by using bundle and surgery techniques. These include connected sums $\#_\ell(S^n\times S^m)$ (see \cite{RW1}) and exotic spheres (see \cite{RW2}). For the condition $scal_k>0$, examples can be constructed using the Gromov--Lawson type surgery result established in \cite{Wo}. See also \cite{CW}, where it is shown that many highly-connected manifolds admit a Riemannian metric of $scal_2>0$.
\smallskip

The main aims of this paper are twofold. Firstly, we provide a far-reaching generalization of a gluing result for positive Ricci curvature due to Perelman \cite{Pe2}. This is Theorem \ref{main} below. Secondly, we use this result to  study the topology of the {\it observer moduli space} of Riemannian metrics with positive intermediate Ricci curvature, and provide the first examples of non-triviality in the higher homotopy groups of such moduli spaces (Theorem~\ref{obs}).

In order to study the topological implications of a curvature condition, a basic question to ask is how this condition behaves with respect to simple topological operations. Perhaps the most basic topological operation is that of gluing together two manifolds with boundary, as for example, one is forced to do when performing the connected sum construction. It is clear that the boundaries being glued must be diffeomorphic, and from a metric point of view, a basic pre-requisite is that the restricted metrics on the two boundary components must be isometric. One then needs to smooth the glued metric in such a way that the curvature condition is preserved. It is unreasonable to expect that this will be possible in general, however Perelman showed that if the sum of the second fundamental forms at the boundary is positive definite, such gluing is possible within positive Ricci curvature. This result has since become an essential tool in the Riemannian geometer's kit. In Theorem \ref{main} we show that Perelman's result can be generalized in a natural way to all the above positive intermediate curvature conditions, from positive scalar curvature up to and including positive sectional curvature.

The topology of spaces and moduli spaces of Riemannian metrics satisfying various types of curvature condition has been a major theme in Riemannian geometry for more than a decade. Most results have focused on spaces (as opposed to moduli spaces) of metrics, as these have generally been more accessible to study. More elusive, and perhaps more desirable, are results about moduli spaces of metrics, i.e.\ the quotient of the space of metrics under the action (by pull-back) of a diffeomorphism group. An important variant of the standard moduli space (the quotient by the full diffeomorphism group) is the observer moduli space (defined below). This offers certain technical advantages over the usual moduli space, and on this basis one can assert its naturality as an object to study.

The topology of the observer moduli space of positive scalar curvature metrics has been shown to be non-trivial in some cases, see for example \cite{BHSW} and \cite{HSS}. The observer moduli space of positive Ricci curvature metrics on spheres has similarly been studied in \cite{BWW}, and in the case of odd-dimensional spheres has been shown to have non-trivial higher homotopy groups. In Theorem \ref{obs}, we propose a result which, to the best of our knowledge, provides the first examples of non-triviality in the higher homotopy groups of the observer moduli spaces for curvature conditions stronger than positive Ricci curvature.
\smallskip

Let us now give the precise details of the results established in this paper. We begin with our gluing result. To state this, we will need to consider the second fundamental form at a boundary of a Riemannian manifold. This is defined by $\II(u,v)=-\langle\nabla_u N,v\rangle$, where $N$ is the inward normal (note that there also exist different conventions in terms of signs), so that the boundary of any geodesic ball in the round sphere that is contained in a hemisphere has positive definite second fundamental form.

The theorem describes the gluing of Riemannian manifolds $M_1,M_2$ with isometric boundaries, and the hypotheses include a condition on the sum $\II_{\partial M_1}+\II_{\partial M_2}$ of the second fundamental forms. We interpret this sum, here and elsewhere in the paper, in the following way. Fix an isometry of the boundaries $\phi:\partial M_1 \to \partial M_2$, and for $u,v$ tangent to $\partial M_1$ set $\bigl(\II_{\partial M_1}+\II_{\partial M_2}\bigr)(u,v)=\II_{\partial M_1}(u,v)+\II_{\partial M_2}(D\phi(u),D\phi(v)).$

\begin{theoremAlph}\label{main}
Let $(M_1, h_1)$ and $(M_2, h_2)$ be a pair of $n$-dimensional Riemannian manifolds with boundary and let $\phi:(\partial M_1, h_1|_{\partial M_1})\rightarrow (\partial M_2, h_2|_{\partial M_2} )$ be an isometry of the boundaries, which we will assume to be compact.
\begin{enumerate}
	\item Suppose that $(M_1,h_1)$ and $(M_2,h_2)$ have $Ric_k>0$ for some $1\le k\le n-1$ and that the sum of second fundamental forms $\II_{\partial M_1}+\II_{\partial M_2}$ is positive semi-definite. Then the
	$C^{0}$-metric $h=h_1\cup_{\phi} h_2$ on the smooth manifold $M_{1}\cup_{\phi}M_{2}$ can be replaced by a $C^{\infty}$-metric with $Ric_k>0$.
	\item Suppose that $(M_1,h_1)$ and $(M_2,h_2)$ have $scal_k>0$ for some $1\le k\le n$ and that the sum of second fundamental forms $\II_{\partial M_1}+\II_{\partial M_2}$ is $k$-non-negative if $k<n-1$ and $(k-1)$-non-negative if $k= n-1,n$. Then the
	$C^{0}$-metric $h=h_1\cup_{\phi} h_2$ on the smooth manifold $M_{1}\cup_{\phi}M_{2}$ can be replaced by a $C^{\infty}$-metric with $scal_k>0$.
\end{enumerate}
Moreover, the smoothed metric can be chosen to agree with the original metrics outside an arbitrarily small neighbourhood of the gluing area.
\end{theoremAlph}

In the case $k=n-1$ in (i) and in the case $k=1$ in (ii) we recover the Perelman gluing theorem, which was also established in \cite[Theorem 7.1]{S}. In the case $k=1$ in (i), i.e.\ gluing two manifolds of positive sectional curvature, the proof of the main theorem in \cite{Ko}, (which asserts that the glued manifold is an Alexandrov space), in fact constructs a $C^1$-metric of positive sectional curvature. Theorem \ref{main} thus provides a smooth version of this result. The special case of doublings also follows from work of Guijarro \cite[Theorem 1.2]{Gu} combined with \cite[Theorem 1.2]{Kr}, in which case the sectional curvatures only need to be non-negative. This case was previously studied for Alexandrov spaces by Perelman \cite{Pe1} and generalized by Petrunin \cite{Pt}. In (i) the case $k=2$ was noted without proof in \cite[Theorem 7.5]{S}. In (ii) the case $k=n$ (positive scalar curvature) was shown in \cite[Theorem 4.11]{BH} and \cite[Theorem 7.2]{S}, and preliminary versions of this result were obtained by Gromov--Lawson \cite[Theorem 5.7]{GL1} and Miao \cite[Corollary 3.1]{Mi}. Note that in this case the condition on the second fundamental forms becomes a condition on the mean curvatures. To the best of our knowledge, for all other values of $k$, the result is new. See also \cite{Ke2} for a survey article on glued spaces with a lower curvature bound.

It was recently shown by Ketterer \cite{Ke1} that the assumptions in (i) of Theorem \ref{main} are necessary in the following sense: suppose as in (i) that $h_1$ and $h_2$ have $Ric_k>0$. Then the $C^0$ metric $h=h_1\cup_\phi h_2$ satisfies the synthetic curvature condition $CD(0,n)$ if and only if the sum of second fundamental forms is positive semi-definite, see \cite[Theorem 1.4 and Corollary 1.9]{Ke1}.

The proof of Theorem \ref{main} will occupy the next section. While part (ii) can be deduced in a similar way to the positive Ricci curvature case, part (i) is substantially more challenging. This is due to the fact that at the gluing area, any Ricci curvature contains a summand of the form $K(\nu, u)$, where $\nu$ is a unit normal to the gluing area and $u$ is tangent to it. The proof in \cite{Pe2} then constructs a smoothing that increases these types of curvatures to any desired value, while keeping the remaining curvatures bounded. For the condition $Ric_k>0$ with $k<n-1$, this strategy no longer works since one needs to consider orthonormal $(k+1)$-frames that might not contain $\nu$ (and that might also not be tangent to the gluing area). A major part of the proof of Theorem \ref{main} consists of analysing and controlling these curvatures.

Examining the proof of Theorem A reveals that there is nothing special about the positivity condition here, in the sense that exactly the same arguments with a curvature lower bound $Ric_k>\kappa$ for any $\kappa\in\R$ will yield an analogous result. Thus we have:

\begin{corollaryAlph}\label{negative_lower_bound}
The above gluing theorem continues to hold if $Ric_k>0$ is replaced by $Ric_k>\kappa$ and $scal_k>0$ is replaced by $scal_k>\kappa$ for any $\kappa\in\R$.
\end{corollaryAlph}

While non-strict curvature inequalities will in general not be preserved by the construction in the proof of Theorem \ref{main}, they are \emph{almost} preserved.

\begin{corollaryAlph}\label{almost-nonnegative}
	In the setting of Theorem \ref{main} suppose that $(M_1,h_1)$ and $(M_2,h_2)$ merely have $Ric_k\geq 0$ in (i) and $scal_k\geq 0$ in (ii). Then the manifold $M_1\cup_\phi M_2$ admits Riemannian metrics of almost non-negative $k$-th intermediate Ricci curvature in (i) and of almost $k$-non-negative Ricci curvature in (ii).
\end{corollaryAlph}
For the precise definitions of almost non-negative $k$-th intermediate Ricci curvature and almost $k$-non-negative Ricci curvature see Definition \ref{D:almost_non-negative} below. This definition coincides with the classical notions of almost non-negative sectional curvature for $k=1$ in (i), almost non-negative Ricci curvature for $k=n-1$ in (i) and $k=1$ in (ii), and almost non-negative scalar curvature for $k=n$ in (ii). Note that this result in particular allows the gluing of two manifolds of non-negative sectional curvature together to produce a metric of almost non-negative sectional curvature.
\smallskip

Turning our attention to moduli spaces, let us recall the notions of {\it observer diffeomorphism group} and {\it observer moduli space} of Riemannian metrics. Given a closed connected manifold $M$ and a choice of basepoint $x_0 \in M,$ the observer diffeomorphism group based at $x_0$, $\text{Diff}_{x_0}(M)$, is the subgroup of the diffeomorphism group $\text{Diff}(M)$ consisting of maps $f \in \text{Diff}(M)$ such that $f(x_0)=x_0$ and $df_{x_0}=\text{id}_{T_{x_0}M}.$ If $\mathcal{R}(M)$ denotes the space of all Riemannian metrics on $M$, equipped with the topology of smooth convergence, then there is a corresponding moduli space $\mathcal{R}(M)/\text{Diff}_{x_0}(M).$ This is the observer moduli space of Riemannian metrics on $M$. Note that $\text{Diff}_{x_0}(M)$ acts on $\mathcal{R}(M)$ by pull-back. The observer moduli space offers certain technical advantages over the full moduli space $\mathcal{R}(M)/\text{Diff}(M)$, which stem from the fact that $\text{Diff}_{x_0}(M)$ acts freely on $\mathcal{R}(M),$ whereas $\text{Diff}(M)$ does not. If we restrict attention to the subspace of Riemannian metrics satisfying some sort of curvature condition invariant under the diffeomorphism group, then we have a corresponding observer moduli space of such metrics. For more background see \cite[chapter 7]{TW}.

Here is our main result about the observer moduli space. To state this, let $\mathcal{M}^{Ric_k>0}_{x_0}(M)$ denote the observer moduli space of metrics on $M$ with $Ric_k>0.$

\begin{theoremAlph}\label{obs}
Given $\ell \in \N$, let $n \ge 8\ell+5$ be odd. Then the observer moduli space of $Ric_k>0$ metrics on $S^n$ satisfies $$\pi_{4\ell}\bigl(\mathcal{M}^{Ric_k>0}_{x_0}(S^n),[ds^2_n]\bigr)\otimes \Q \neq 0$$ for all $(n+3)/2 \le k \le n-1.$ Here, $ds_n^2$ denotes the round metric of radius $1$.
\end{theoremAlph}

This result is a direct generalization of the main theorem in \cite{BWW}, which deals with the case of positive Ricci curvature, (i.e. the case $k=n-1$). Note that the relationship between $n$ and $\ell$ appearing in Theorem \ref{obs} was not explicitly considered in \cite{BWW}. The proof of Theorem \ref{obs} follows the Ricci positive argument closely.

\smallskip

Finally, by considering doublings of manifolds, we derive the following consequences of Theorem~\ref{main}.
\begin{corollaryAlph}\label{tot_geodesic}
	Let $(M^n,g)$ be a Riemannian manifold of $Ric_k>0$ (resp.\ $scal_k>0$) with compact boundary $\partial M$ such that the second fundamental form $\II_{\partial M}$ is positive semi-definite (resp.\ $k$-non-negative if $k<n-1$ and $(k-1)$-non-negative if $k=n-1,n$). Then $M$ admits a Riemannian metric of $Ric_k>0$ (resp.\ $scal_k>0$) such that the boundary is totally geodesic, i.e.\ $\II_{\partial M}\equiv 0$.
\end{corollaryAlph}

\begin{corollaryAlph}\label{double}
	Let $n\in\N$, $1\leq d\leq n-4$, and $k\geq n-d$ if $d>1$ and $k=n$ if $d=1$. Then for any $n$-dimensional compact manifold $M^n$ such that the inclusion $\partial M\hookrightarrow M$ is $d$-connected, the double $M\cup_{\partial M}M$ admits a metric of $scal_k>0$. In particular, if $M$ is a closed, $d$-connected $n$-dimensional manifold, then $M\#(-M)$ admits a metric of $scal_k>0$.
\end{corollaryAlph}
The condition $scal_k>0$ can be preserved under connected sums for all $k\geq 2$ by \cite[Theorem 2.1]{Wo}, see also \cite[item (5) on p. 15]{Ho}. Note however, that we do not assume in Corollary \ref{double} that $M$ itself admits a metric of $scal_k>0$. The case $k=n$ and $d=1$, i.e.\ the case where $M$ has positive scalar curvature and the inclusion of the boundary is 1-connected, already follows from a result of Lawson--Michelsohn \cite[Theorem 1.1]{LM}, (which asserts that $M$ admits a Riemannian metric of positive sectional curvature and positive mean curvature on the boundary), in combination with the gluing result for doubles by Gromov--Lawson \cite[Theorem 5.7]{GL1}. The general case can also be proved by using bordism techniques and the surgery results for $scal_k>0$ of Gromov--Lawson \cite{GL2} (for $k=n$) and of Wolfson \cite{Wo} (for general $k$), see Remark \ref{bordism} below. We note that the metric constructed in the proof of Corollary \ref{double} has the additional property that it has in fact positive sectional curvature outside an arbitrarily small neighbourhood of the gluing area and the metric is invariant under the $\Z/2$ action that interchanges the two copies of $M$. In particular, the glued boundary $\partial M$ is totally geodesic in $M\cup_{\partial} M$.

Theorem \ref{main} can also be used for connected sum constructions for metrics of $Ric_k>0$ similar to Perelman's original applications in \cite{Pe2} for positive Ricci curvature, and Burdick's extensions and generalizations in \cite{Bu1,Bu2,Bu3}. This is the content of the article \cite{RW3}.

This paper is organized as follows. In Section \ref{main_proof} we prove Theorem \ref{main} by defining an explicit metric near the gluing area to interpolate between the original metrics. The construction of the metric is partially based on the original construction in \cite{Pe2}, in that it initially uses a $C^1$-spline interpolation. However, in the case of $Ric_k>0$, ($k<n-1$), the analysis of the curvature properties requires additional work, which will be carried out in Subsection \ref{C1}. The smoothing of the metric from $C^1$ to $C^\infty$ differs from the Perelman approach, and is the content of Subsection \ref{C_infty}, and in Subsections \ref{almost-nonnegative_subsec} and \ref{doubling_corollaries} we prove Corollaries \ref{almost-nonnegative} and \ref{tot_geodesic}, \ref{double}, respectively. Finally, in Section \ref{obs_proof} we prove Theorem \ref{obs}.

\begin{ack}
	The authors would like to thank the anonymous referees for their very careful reading of the paper and helpful suggestions. The exposition is much improved as a consequence.
\end{ack}

\section{Proof of Theorem \ref{main}}\label{main_proof}

As discussed in the introduction, the set-up is as follows: we have two manifolds with boundary $M_1,M_2$, equipped with Riemannian metrics satisfying $Ric_k>0$ for some $1 \le k \le n-1$ or $scal_k>0$ for some $1\leq k\leq n$. We assume that $\partial M_1$ is isometric to $\partial M_2.$ In fact, as explained in \cite[\S2]{BWW}, it suffices to consider the case where $\partial M_1=X=\partial M_2$ and the isometry is the identity. We will make this assumption throughout the sequel.

In \cite{Pe2}, Perelman sketches the proof of his Ricci positive gluing result, giving only minimal details. By gluing the manifolds and metrics along the common boundary, one creates a $C^0$ metric overall: $C^0$ along the gluing and smooth elsewhere. The idea is then to delete the metric in a small neighbourhood of $X$, and to replace it with a cubic interpolation to create a $C^1$ metric overall: $C^1$ at the two joining hypersurfaces and smooth elsewhere. Perelman's next step is to smooth the $C^1$ joins using a small quintic interpolation to create a $C^2$ metric. From there, passing to a $C^\infty$ metric preserving the $Ric>0$ condition is straightforward.

In \cite[\S2.2]{BWW} the precise details of this construction are presented. The principal difficulty with taking these ideas and proving an analogous result about positive intermediate Ricci curvatures is the problem of detecting and controlling such curvatures. Consequently, the proof of Theorem A requires some detailed and delicate analysis. Our argument, in fact, only follows the Perelman plan as far as creating a $C^1$ metric join. We then take a different, and arguably more efficient approach to smoothing the $C^1$ metric to a metric of class $C^\infty$. This involves mollification techniques. 

We will proceed as follows. Given a pair of Riemannian $n$-manifolds $(M_1,h_1),(M_2,h_2)$ with $Ric_k>0$ or $scal_k>0$, assume that $(\partial M_1,h_1|_{\partial M_1})=(\partial M_2,h_2|_{\partial M_2}).$ We first apply a small deformation to $h_1$ near $\partial M_1$ as in \cite[Proposition 1.2.11]{Bu1} that leaves the metric on $\partial M_1$ unchanged and slightly increases the second fundamental form on $\partial M_1$, so that we can assume that the assumption on $\II_{\partial M_1}+\II_{\partial M_2}$ is strictly satisfied. Now glue the boundaries via the identity map to create a smooth manifold $M_1 \cup M_2$ with $C^0$ Riemannian metric $h_1 \cup h_2$. We can identify a neighbourhood of the gluing with the product $X \times [-\epsilon,\epsilon]$ for some small $\epsilon>0,$ where the second factor encodes the normal parameter with respect to the $C^0$ metric.
\bigskip

\subsection{The $C^1$-smoothing}\label{C1}

Our first task is to replace the metric on this neighbourhood with a new metric $g$ which joins with $h_1$ for $t<-\epsilon$ and with $h_2$ for $t>\epsilon$ to give a $C^1$-metric on $M_1 \cup M_2$. It will be convenient to work in a slightly larger neighbourhood, $X \times [-\epsilon-\iota,\epsilon+\iota]$, for some small $\iota>0$. The metric $g$ will take the form
$g=dt^2+g_t$ in this neighbourhood. Thus $g$ agrees with $h_1$ for $t \in [-\epsilon-\iota,-\epsilon],$ and with $h_2$ for $t \in [\epsilon,\epsilon+\iota].$ Denote by $h_i(t)$, $i=1,2$, the metric induced by $h_i$ on the hypersurface with constant distance $t$ from $X$, so locally we can write $h_i=dt^2+h_i(t).$ It is easy to check (see \cite[\S2.2]{BWW}) that setting $g_t$ for $t \in [-\epsilon,\epsilon]$ to be the following cubic expression in $t$ will create the desired $C^1$-join:
\begin{equation}\label{eq:g(t)}
\begin{array}{lcl}
  g_t&=& \displaystyle
  \frac{t+\epsilon}{2\epsilon}h_2(\epsilon)
  - \frac{t-\epsilon}{2\epsilon}h_1(-\epsilon)
   + \frac{(t-\epsilon)^2(t+\epsilon)}{4\epsilon^2}
 \left[h_1'(-\epsilon)-\frac{1}{2\epsilon}[h_2(\epsilon)-h_1(-\epsilon)]\right]
\\
  \\
  & & \displaystyle + \frac{(t+\epsilon)^2(t-\epsilon)}{4\epsilon^2}
 \left[h_2'(\epsilon)-\frac{1}{2\epsilon}[h_2(\epsilon)-h_1(-\epsilon)]\right].
\end{array}
\end{equation}

Note that dashes always denote $t$-direction derivatives.
(Throughout this section, all quantities under discussion will have $C^2$-dependence on the metrics $h_1,h_2$. For simplicity, this dependency will often be suppressed in the notation used.)

From the product structure of $X\times[-\epsilon,\epsilon]$ we obtain a splitting $T(X\times[-\epsilon,\epsilon])\cong TX\times T[-\epsilon,\epsilon]$, and with respect to this splitting, for any $t\in[-\epsilon,\epsilon]$, tangent vectors $u,v$ tangent to $X$ define tangent vectors $u_t,v_t\in T(X\times [-\epsilon,\epsilon])$ tangent to $X\times\{t\}$. In practice we will often suppress the subscript, since it will be clear from the context which slice $u,v$ are tangent to. However, in the case where $t=\pm \epsilon$, we will usually write $u_{\pm\epsilon},v_{\pm\epsilon}$ to emphasize the location of these vectors. Similarly, any (local) tangent vector field on $X$ defines a (local) tangent vector field on $X\times [-\epsilon,\epsilon]$ which is tangent to $X\times \{t\}$ for all $t\in[-\epsilon,\epsilon]$. Note that for $t=0$ the vector $\partial_t$ is the outward normal w.r.t.\ $h_1$ and the inward normal w.r.t.\ $h_2$. For functions $a_\epsilon,b_\epsilon\colon [-\epsilon,\epsilon]\to V$ defined for all $\epsilon$ sufficiently small, with $(V,\lVert\cdot\rVert)$ a normed vector space, we say that $a_\epsilon=b_\epsilon+O(\epsilon^m)$ as $\epsilon\to 0$ if
\[ \max_{t \in [-\epsilon,\epsilon]}\lVert a_\epsilon(t)-b_\epsilon(t)\rVert=O(\epsilon^m) \]
as $\epsilon\to 0$ in the usual sense. In the following, $V$ will denote the space of symmetric $(0,2)$-tensors on $X$, equipped with the induced norm of a fixed background metric on $X$.

    \begin{lemma}[{cf.\ also \cite[Proposition 9]{Bu3} and \cite[Lemma 4]{BWW}}]
		\label{g_t&g_t'&g_t''}
		We have
		\begin{enumerate}
			\item $g_t=h_i(0)+O(\epsilon)$ as $\epsilon\to 0$ for $i=1,2$,
			\item $g_t'=\bigl(\frac{\epsilon-t}{2\epsilon}h_1'(0)+ \frac{\epsilon+t}{2\epsilon}h_2'(0)\bigr)+O(\epsilon)$ as $\epsilon\to 0$, and
			\item $g_t''=\frac{1}{2\epsilon}(h_2'(0)-h_1'(0))+O(1)$ as $\epsilon\to 0$.
		\end{enumerate}
	\end{lemma}
	In other words, up to a summand in $O(\epsilon)$, $g_t$ and $\epsilon g_t''$ are constant and $g_t'$ interpolates linearly between $h_1'(0)$ and $h_2'(0)$.
	
	\begin{proof}
		\begin{enumerate}
			\item By definition of $g_t$ we need to consider the following expression:
			\begin{align}
				\notag g_t-h_i(0)=& \frac{t}{2\epsilon}(h_2(\epsilon)-h_1(-\epsilon))+\frac{1}{2}(h_2(\epsilon)+h_1(-\epsilon))-h_i(0)\\
				\notag  &+ \frac{(t-\epsilon)^2(t+\epsilon)}{4\epsilon^2}
				\left[h_1'(-\epsilon)-\frac{1}{2\epsilon}[h_2(\epsilon)-h_1(-\epsilon)]\right]\\
				\label{g_t} &+ \frac{(t+\epsilon)^2(t-\epsilon)}{4\epsilon^2}
				\left[h_2'(\epsilon)-\frac{1}{2\epsilon}[h_2(\epsilon)-h_1(-\epsilon)]\right]. 
			\end{align}

By Taylor's theorem and the compactness of $X$, we have
\begin{align}
    \notag\frac{h_2(\epsilon)-h_1(-\epsilon)}{\epsilon}&=\frac{h_2(0)+h_2'(0)\epsilon+O(\epsilon^2)-(h_1(0)-h_1'(0)\epsilon+O(\epsilon^2))  }{\epsilon}\\
    &=h_2'(0)+h_1'(0)+O(\epsilon)\label{EQ:h_2-h_1}
\end{align}
as $\epsilon \to 0$. Similarly, we have
\begin{align*}
    \frac{1}{2}\left(h_2(\epsilon)+h_1(-\epsilon)\right)-h_i(0)&=\frac{1}{2}\left( h_2(0)+h_1(0)+O(\epsilon) \right)-h_i(0)\\
    &=O(\epsilon)
\end{align*}
as $\epsilon\to 0$. As $|t| \le \epsilon,$ we deduce that the whole first line of \eqref{g_t} is $O(\epsilon)$.
By the same analysis, we have
\[h_i'(\epsilon)-\frac{1}{2\epsilon}[h_2(\epsilon)-h_1(-\epsilon)]=h_i'(0)-\frac{1}{2}\left(h_2'(0)+h_1'(0)\right)+O(\epsilon)\]
as $\epsilon\to 0$. It follows that $h_i'(\epsilon)-\frac{1}{2\epsilon}[h_2(\epsilon)-h_1(-\epsilon)]$ remains bounded as $\epsilon\to 0$.	Finally, since $|t|\leq \epsilon$, we have $|(t\pm\epsilon)^2(t\mp\epsilon)|=O(\epsilon^3)$ as $\epsilon\to 0$. Hence, the second and third line of \eqref{g_t} are both in $O(\epsilon)$ as $\epsilon\to 0$.
			
			\item A calculation shows that $g'_t$ is given by
			\begin{align}
				\notag g_t'= & \frac{1}{2\epsilon}(h_2(\epsilon)-h_1(-\epsilon))+\frac{  2(t^2-\epsilon^2)+(t-\epsilon)^2}{4\epsilon^2}\left[ h_1'(-\epsilon)-\frac{1}{2\epsilon}[h_2(\epsilon)-h_1(-\epsilon)]  \right]\\
				\notag &+\frac{2(t^2-\epsilon^2)+(t+\epsilon)^2}{4\epsilon^2}\left[h_2'(\epsilon)-\frac{1}{2\epsilon}[h_2(\epsilon)-h_1(-\epsilon)]  \right]\\
				\notag =& \frac{1}{2\epsilon}(h_2(\epsilon)-h_1(-\epsilon))+\frac{3t^2-\epsilon^2}{4\epsilon^2}\left[ h_1'(-\epsilon)+h_2'(\epsilon)-\frac{1}{\epsilon}[h_2(\epsilon)-h_1(-\epsilon)] \right]\\
				\label{g_t'} &+\frac{t}{2\epsilon}\left[ h_2'(\epsilon)-h_1'(-\epsilon) \right]  .
			\end{align}
            By Taylor's theorem, we have
            \begin{align}
                \notag\frac{1}{\epsilon}\bigl[ h_1'(-\epsilon)+h_2'(\epsilon)&-\frac{1}{\epsilon}[h_2(\epsilon)-h_1(-\epsilon)] \bigr]\\
                \notag=&\frac{1}{\epsilon}\bigl[  h_1'(0)-\epsilon h_1''(0)+h_2'(0)+\epsilon h_2''(0)+O(\epsilon^2)\\
                \notag&-\frac{1}{\epsilon}[ h_2(0)+\epsilon h_2'(0)+\tfrac{1}{2}\epsilon^2 h_2''(0)-h_1(0)+\epsilon h_1'(0)-\tfrac{1}{2}\epsilon^2 h_1''(0)+O(\epsilon^3) ] \bigr]\\
                =&\frac{1}{2}(h_2''(0)-h_1''(0))+O(\epsilon).\label{EQ:l'H}
            \end{align}
            as $\epsilon\to 0$. Hence, since $|t|\leq\epsilon$, the second term in \eqref{g_t'} is in $O(\epsilon)$ as $\epsilon\to 0$. Further, by \eqref{EQ:h_2-h_1} and since $h_2'(\epsilon)=h_2'(0)+O(\epsilon)$ and $h_1'(-\epsilon)=h_1'(0)+O(\epsilon)$ as $\epsilon\to 0$, the claim follows.
			\item Finally, we compute
			\begin{align*}
				g_t''=\frac{3t}{2\epsilon^2}\left[ h_1'(-\epsilon)+h_2'(\epsilon)-\frac{1}{\epsilon}\left[h_2(\epsilon)-h_1(-\epsilon) \right] \right]+\frac{1}{2\epsilon}\left[h_2'(\epsilon)-h_1'(-\epsilon) \right].
			\end{align*}
			Similarly as before, by using \eqref{EQ:l'H}, we see that the first term is in $O(1)$ as $\epsilon\to 0$, so the claim follows.
		\end{enumerate}
	\end{proof}
	\begin{lemma}\label{II}
		For any metric on $X\times[-\epsilon,\epsilon]$ of the form $dt^2+g_t$, where $g_t$ is a metric on $X$, we have
		\[ \II_t=-\frac{1}{2}g_t', \]
		where $\II_t$ denotes the second fundamental form of the hypersurface $\{t\}\times X$.
	\end{lemma}
	\begin{proof}
		Let $u,v$ be local vector fields on $X$, which we extend to local vector fields on $X\times[-\epsilon,\epsilon]$ as described above. Then we have $[\partial_t,u]=[\partial_t,v]=0$ by the product structure of $X\times[-\epsilon,\epsilon]$. Hence, we obtain
		\[\II_t(u,v)=-g(\nabla_u\partial_t,v)=-g(\nabla_{\partial_t} u,v)=-\partial_t g(u,v)+g(u,\nabla_{\partial_t}v)=-g_t'(u,v)-\II_t(v,u). \] 
		Since $\II$ is symmetric, this gives $\II(u,v)=-g'_t(u,v)/2$.
	\end{proof}
	Denote by $K(u,v)$ the sectional curvature of the plane spanned by $u,v$ in $X\times [-\epsilon,\epsilon]$ and let $K_t(u,v)$ denote the corresponding sectional curvature of $X\times\{t\}$ with respect to the induced metric. Set $$\phi_t(u,v)=g_t'(u,u)g_t'(v,v)-g_t'(u,v)^2,$$ and $$\psi_t(u,v)=4\bigl(g_t(u,u)g_t(v,v)-g_t(u,v)^2\bigr).$$
	\begin{lemma}\label{K}
		For linearly independent $u,v$ tangent to $X$ we have
		\begin{enumerate}
			\item $K(u_t,v_t)=K_t(u,v)-\frac{\phi_t(u,v)}{\psi_t(u,v)}$ and $K_{t}(u,v)=K_{h_i(0)}(u,v)+O(\epsilon)$ as $\epsilon\to 0$,
			\item $R(u_t,\partial_t,\partial_t,u_t)= \frac{1}{4\epsilon}(h_1'(0)-h_2'(0))(u,u) +O(1)$ as $\epsilon\to 0$,
		\end{enumerate}
        where the $O(1)$- and $O(\epsilon)$-bounds are independent of $u$ and $v$.
	\end{lemma}
	\begin{proof}
		For (i) we apply the Gauss formula (see for example \cite[Theorem 2.2]{Ch}):
		$$K(u,v)=K_t(u,v)-\frac{\II_t(u,u)\II_t(v,v)-\II_t(u,v)^2}{g_t(u,u)g_t(v,v)-g_t(u,v)^2}.$$
		The first claim now follows from Lemma \ref{II}.

        The second statement of (i) is essentially a consequence of the proof of Lemma \ref{g_t&g_t'&g_t''}(i). The only derivatives appearing in this argument are $t$-derivatives, and one simply observes that the argument continues to hold verbatim if one applies any $X$-direction derivative(s) to all of the terms. In particular, if ${\mathcal D}_X$ denotes any first or second order $X$-direction derivative, we have ${\mathcal D}_Xg_t={\mathcal D}_X h_i(0)+O(\epsilon).$ As the (intrinsic) sectional curvatures $K_t$ of the metrics $g_t$ depend only on such derivatives, the claim follows automatically.

        For (ii), note that by \cite[equation (2-5)]{BWW} and the subsequent line, we have $$R(u_t,\partial_t,\partial_t,u_t)=-\frac{1}{2}g''_t(u,u)+g_t(S_t(u),S_t(u)),$$ where $S_t$ is the shape operator of the hypersurface $X \times \{t\}$. By Lemma \ref{II}, $S_t$ has a first-order dependence on the metric $g_t$, and hence it follows from Lemma \ref{g_t&g_t'&g_t''} parts (i) and (ii) that the dependence of $g_t(S_t(u),S_t(u))$ on $\epsilon$ is $O(1)$. The claim now follows immediately from Lemma \ref{g_t&g_t'&g_t''}(iii).
	\end{proof}

	The following proposition will be important in the proof of (i) of Theorem \ref{main}.
	
	\begin{proposition}\label{K_lower_bound}
		Suppose that $h_1'(0)-h_2'(0)$ is positive definite. Then, for linearly independent $u,v$ tangent to $X$, we have
		\[ K_{g}(u_t,v_t)\geq \frac{\epsilon-t}{2\epsilon}K_{h_1}(u_0,v_0)+\frac{\epsilon+t}{2\epsilon}K_{h_2}(u_0,v_0)+O(\epsilon) \]
		as $\epsilon \to 0$, where the $O(\epsilon)$-bound is independent of $u$ and $v$.
	\end{proposition}
	\begin{proof}
        Without loss of generality, we can assume that $u$ and $v$ are $h_i(0)$-orthonormal.
		We first consider the quotient $\frac{\phi_t(u,v)}{\psi_t(u,v)}$. We set 
		\[\overline{\psi}(u,v)=4(h_i(u,u)h_i(v,v)-h_i(u,v)^2)=4 \]
		where we have written $h_i(u,u)$ as shorthand for $h_i(0)(u,u)$, etc. Note that this definition is independent of $i\in\{1,2\}$. Then, by Lemma \ref{g_t&g_t'&g_t''}, we have
		\[ \psi_t(u,v)=\overline{\psi}(u,v)+O(\epsilon) \]
		as $\epsilon\to 0$ and the $O(\epsilon)$-bound is independent of $u$ and $v$. Hence,
		\[ \frac{1}{\psi_t(u,v)}=\frac{1}{\overline{\psi}(u,v)}+\frac{\overline{\psi}(u,v)-\psi_t(u,v)}{\psi_t(u,v)\overline{\psi}(u,v)}=\frac{1}{\overline{\psi}(u,v)}+O(\epsilon) \]
		as $\epsilon\to 0$.
		
		For $\phi_t(u,v)$, if we set
		\[\overline{h}'_t=\frac{\epsilon-t}{2\epsilon}h_1'(0)+\frac{\epsilon+t}{2\epsilon}h_2'(0), \]
		then it follows from Lemma \ref{g_t&g_t'&g_t''} that
		\[ \phi_t(u,v)=\overline{\phi}_t(u,v)+O(\epsilon), \]
		where 
		\[ \overline{\phi}_t(u,v)=\overline{h}_t'(u,u)\overline{h}_t'(v,v)-\overline{h}_t'(u,v)^2 \]
        and the $O(\epsilon)$-bound is again independent of $u$ and $v$.
        
		Note that $\overline{\phi}_t(u,v)$ is a polynomial in $t$ of degree $2$. Its derivative at $t=\pm\epsilon$ is given by
		\begin{align*}
			\overline{\phi}_{\pm\epsilon}'(u,v)=&\frac{1}{2\epsilon}(h_2'(u,u)-h_1'(u,u))h_i'(v,v)+\frac{1}{2\epsilon}h_i'(u,u)(h_2'(v,v)-h_1'(v,v))\\
			&-\frac{1}{\epsilon}(h_2'(u,v)-h_1'(u,v))h_i'(u,v),
		\end{align*}
		where $i=1$ for $t=-\epsilon$ and $i=2$ for $t=\epsilon$. It follows that
		\begin{align*}
			\epsilon(\overline{\phi}'_\epsilon(u,v)-\overline{\phi}'_{-\epsilon}(u,v))&= (h_1'(u,u)-h_2'(u,u))(h_1'(v,v)-h_2'(v,v))-(h_1'(u,v)-h_2'(u,v))^2\\
			&=\det\begin{pmatrix}
				h_1'(u,u)-h_2'(u,u) & h_1'(u,v)-h_2'(u,v) \\
				h_1'(u,v)-h_2'(u,v) & h'_1(v,v)-h_2'(v,v) \\
			\end{pmatrix}.
		\end{align*}
		Since, by assumption, $h_1'(0)-h_2'(0)$ is positive definite, its restriction to the $2$-plane spanned by $u$ and $v$ is positive definite. Hence, the last expression is positive for all $h_i(0)$-orthonormal $u$ and $v$. Thus, as a polynomial of degree 2 with increasing derivative, $\overline{\phi}_t(u,v)$ is a convex function in $t$.
		
		We now consider the quotient $\frac{\overline{\phi}_t(u,v)}{\overline{\psi}(u,v)}$. Since $\overline{\psi}$ does not depend on $t$, this quotient is also a convex function in $t$. Thus, we have shown that
		\[ \frac{\phi_t(u,v)}{\psi_t(u,v)}=\frac{\overline{\phi}_t(u,v)}{\overline{\psi}(u,v)}+O(\epsilon) \] 
		for $\epsilon\to 0$, where the upper bound in $O(\epsilon)$ is independent of $u$ and $v$, and where $\frac{\overline{\phi}_t(u,v)}{\overline{\psi}(u,v)}$ is convex for all $h_i$-orthonormal $u$ and $v$. It follows from the convexity property that 
		\begin{align*}
			\frac{\overline{\phi}_t(u,v)}{\overline{\psi}(u,v)}\leq& \Bigl(\frac{\epsilon-t}{2\epsilon}\Bigr)\frac{\overline{\phi}_{-\epsilon}(u,v)}{\overline{\psi}(u,v)}+\Bigl(\frac{\epsilon+t}{2\epsilon}\Bigr)\frac{\overline{\phi}_{\epsilon}(u,v)}{\overline{\psi}(u,v)}\\ \\
			=& \Bigl(\frac{\epsilon-t}{2\epsilon}\Bigr)\frac{h_1'(u,u)h_1'(v,v)-h_1'(u,v)^2}{4(h_1(u,u)h_1(v,v)-h_1(u,v)^2)}\\ 
			&+\Bigl(\frac{\epsilon+t}{2\epsilon}\Bigr)\frac{h_2'(u,u)h_2'(v,v)-h_2'(u,v)^2}{4(h_2(u,u)h_2(v,v)-h_2(u,v)^2)},\\ \\
			=& \Bigl(\frac{\epsilon-t}{2\epsilon}\Bigr)\frac{\II_{\partial M_1}(u,u)\II_{\partial M_1}(v,v)-\II_{\partial M_1}(u,v)^2}{h_1(u,u)h_1(v,v)-h_1(u,v)^2}\\
			&+\Bigl(\frac{\epsilon+t}{2\epsilon}\Bigr)\frac{\II_{\partial M_2}(u,u)\II_{\partial M_2}(v,v)-\II_{\partial M_2}(u,v)^2}{h_2(u,u)h_2(v,v)-h_2(u,v)^2}
		\end{align*}
		where each metric and derivative is considered at $t=0$ and $\II_{\partial M_i}$ is taken with respect to $h_i$. Hence, by Lemma \ref{K}, it follows that
		\begin{align*}
			K(u_t,v_t)&=K_t(u,v)-\frac{\phi_t(u,v)}{\psi_t(u,v)}\\ \\
			=&K_{h_i(0)}(u,v)-\frac{\overline{\phi}_t(u,v)}{\overline{\psi}(u,v)}+O(\epsilon)\\ \\
			\geq& K_{h_i(0)}(u,v)-\Bigl(\frac{\epsilon-t}{2\epsilon}\Bigr)\frac{\II_{\partial M_1}(u,u)\II_{\partial M_1}(v,v)-\II_{\partial M_1}(u,v)^2}{h_1(u,u)h_1(v,v)-h_1(u,v)^2}\\
			&-\Bigl(\frac{\epsilon+t}{2\epsilon}\Bigr)\frac{\II_{\partial M_2}(u,u)\II_{\partial M_2}(v,v)-\II_{\partial M_2}(u,v)^2}{h_2(u,u)h_2(v,v)-h_2(u,v)^2}+O(\epsilon)\\ \\
			=&\Bigl(\frac{\epsilon-t}{2\epsilon}\Bigr)\left(K_{h_1(0)}(u,v)-\frac{\II_{\partial M_1}(u,u)\II_{\partial M_1}(v,v)-\II_{\partial M_1}(u,v)^2}{h_1(u,u)h_1(v,v)-h_1(u,v)^2}\right)\\
			&+\Bigl(\frac{\epsilon+t}{2\epsilon}\Bigr)\left(K_{h_2(0)}(u,v) -\frac{\II_{\partial M_2}(u,u)\II_{\partial M_2}(v,v)-\II_{\partial M_2}(u,v)^2}{h_2(u,u)h_2(v,v)-h_2(u,v)^2}\right)+O(\epsilon)\\ \\
			=&\Bigl(\frac{\epsilon-t}{2\epsilon}\Bigr)K_{h_1}(u_0,v_0)+\Bigl(\frac{\epsilon+t}{2\epsilon}\Bigr)K_{h_2}(u_0,v_0)+O(\epsilon).
		\end{align*}
	\end{proof}
	
	\begin{corollary}\label{Ric_k-interpolation}
		Suppose that $h_1'(0)-h_2'(0)$ is positive definite. Let $(v,e^1,\dots,e^k)$ be a $(k+1)$-frame of linearly independent vectors tangent to $X$. Then
		\[ \sum_{i=1}^{k}K_g(v_t,e^i_t)\geq \frac{\epsilon-t}{2\epsilon}\sum_{i=1}^{k}K_{h_1}(v_0,e^i_0)+\frac{\epsilon+t}{2\epsilon}\sum_{i=1}^{k}K_{h_2}(v_0,e^i_0)+O(\epsilon) \]
		as $\epsilon\to 0$.
	\end{corollary}

Before proving Theorem \ref{main} we need a definition:
\begin{definition}\label{eta-close}
For $\eta>0$ we will say that a $(k+1)$-frame $\{u_0,u_1,\dots,u_k\}$ in a Riemannian manifold $(M,\langle\,,\,\rangle)$ is $\eta$-nearly orthonormal if $\langle u_i,u_i\rangle \in [1-\eta,1+\eta]$ for all $i$, and $|\langle u_i,u_j \rangle|\le \eta$ whenever $i \neq j.$
\end{definition}

Suppose that $\{v,e^1,\dots,e^k\}$ is an orthonormal $(k+1)$-frame tangent to $X$ at some point $x\in X$. As before, by using the product structure, we extend it to a frame field along $\{x\}\times [-\epsilon,\epsilon]$. The resulting frame field will not in general be orthonormal, however by (i) of Lemma \ref{g_t&g_t'&g_t''} and the compactness of the set of all such frame fields we have:

\begin{observation}\label{obs1}

Given $\epsilon>0$ there is an $\eta=\eta(\epsilon,h_1,h_2)>0$, with $\eta \to 0$ as $\epsilon \to 0,$ such that every $g$-orthonormal frame field tangent to $X \times \{t\}$ is $\eta$-nearly orthonormal for $h_i(0)$.

\end{observation}

We now prove Theorem \ref{main}. We start with case (i).

\begin{proposition}\label{C^1_result}
Given that $h_1$ and $h_2$ both have $Ric_k>0$ and $h_1'(0)-h_2'(0)$ is positive definite, for all $\epsilon$ sufficiently small the $C^1$-interpolating metric $g=dt^2+g_t$ on $X \times [-\epsilon,\epsilon]$ has $Ric_k>0$.
\end{proposition}

\begin{proof}

Since $X$ is compact and the sectional curvature $K_{h_i}(u_0,v_0)$ is continuous in $u$ and $v$, there exists $\eta=\eta\bigl(h_i(0)\bigr)>0$ such that every $\eta$-nearly orthonormal $(k+1)$-frame for $(X,h_i(0))$ has $Ric_k>0$ with respect to both $h_1$ and $h_2$. Here we extended the definition of $Ric_k>0$ to $\eta$-nearly orthonormal frames in the obvious way (where we choose $\eta$ sufficiently small so that any $\eta$-nearly orthonormal frame is linearly independent). Moreover, by possibly choosing $\eta$ smaller, we can assume that $h_1'(0)-h_2'(0)$ is positive on $\eta$-unit vectors (i.e.\ vectors $v$ with $\lVert v\rVert \in[1-\eta,1+\eta]$). We denote a positive lower bound by $C>0$.
	
By Observation \ref{obs1} there exists $\epsilon_0>0$, such that every $g$-orthonormal frame tangent to $X\times \{t\}\subseteq X\times [-\epsilon,\epsilon]$ is $\eta$-nearly orthonormal for $h_i(0)$ for all $\epsilon<\epsilon_0$. It follows that every orthonormal $(k+1)$-frame for $(X\times\{t\},g_t )$ has $Ric_k>0$ with respect to $h_i$ when considered as a frame tangent to $X\times\{0\}$.

Let $\{v,e^1,\dots,e^k \}$ be a $g$-orthonormal frame. Each $K_g(v,e^i)$ only depends on the plane $P_i$ spanned by $v$ and $e^i$. It will be convenient to re-choose the basis for $P_i$ as follows. In the case where $P_i$ is not tangent to $X \times \{t\}$, observe that $T_{(x,t)}\bigl(X \times \{t\}\bigr) \cap P_i$ is a one-dimensional space. It follows that there exists a $g$-unit vector $w^i$ in this intersection. Let $u^i$ be a $g$-unit vector in $P_i$ which is $g$-orthogonal to $w^i$. By assumption $u^i$ has a non-zero component in the $t$-direction, so we can write $u^i=\cos(\theta_i)\tilde{u}^i+\sin(\theta_i)\partial_t$ for some angle $\theta_i$ and a $g$-unit vector $\tilde{u}^i$ tangent to $X \times \{t\}.$ Notice that $\tilde{u}^i$ is orthogonal to $w^i$. 
Moreover, by rechoosing $\tilde{u}^i$ to be its negative if necessary, we can assume that $\theta_i \in (0,\pi/2].$ On the other hand if $P_i$ {\it is} tangent to $X \times \{t\}$, set $u^i=\tilde{u}^i=v$ and $w^i=e^i$. In this case we have $\theta_i=0$. In all cases the sectional curvature of $P_i$ with respect to $g$ satisfies
\begin{align*}
	K(P_i)&=K(u^i,w^i) \\
	&=K\bigl(\cos(\theta_i)\tilde{u}^i+\sin(\theta_i)\partial_t,w^i\bigr) \\
	&=R\bigl(\cos(\theta_i)\tilde{u}^i+\sin(\theta_i)\partial_t,w^i,w^i,\cos(\theta_i)\tilde{u}^i+\sin(\theta_i)\partial_t\bigr) \\
	&=\cos^2(\theta_i)R(\tilde{u}^i,w^i,w^i,\tilde{u}^i)+\sin^2(\theta_i)R(\partial_t,w^i,w^i,\partial_t)+\sin(2\theta_i)R(\tilde{u}^i,w^i,w^i,\partial_t) \\
	&=\cos^2(\theta_i)K(\tilde{u}^i,w^i)+\sin^2(\theta_i)R(\partial_t,w^i,w^i,\partial_t)+\sin(2\theta_i)R(\tilde{u}^i,w^i,w^i,\partial_t).
\end{align*}
In \cite[\S2]{BWW} it is shown that the term $R(\tilde{u}^i,w_i,w_i,\partial_t)$ appearing above is bounded independent of $\epsilon$ as $\epsilon \to 0$, assuming the vectors $\tilde{u}^i,w_i,\partial_t$ are fixed. Since the pair $\{\tilde{u}^i,w_i\}$ is $g$-orthonormal, as $\epsilon \le \epsilon_0$ we have that  $\{\tilde{u}^i,w_i\}$ is $\eta$-nearly orthonormal for $h_i(0).$ Clearly, by compactness of the set of choices, $\{R(v,y,y,\partial_t)\,|\,\{v,y\}\, \text{ is } \eta\text{-nearly orthonormal for }h_i(0)\}$ is bounded independent of $\epsilon$.

	Hence, using Lemma \ref{K} and Proposition \ref{K_lower_bound}, we can now estimate:
	\begin{align*}
		K(P_i)\geq & \cos^2(\theta_i)\left( \frac{\epsilon-t}{2\epsilon}K_{h_1}(\tilde{u}^i_0,w^i_0)+\frac{\epsilon+t}{2\epsilon}K_{h_2}(\tilde{u}^i_0,w^i_0)+O(\epsilon) \right)\\
		&+\sin^2(\theta_i)\left( \frac{1}{2\epsilon}(h_1'(0)-h_2'(0))(w^i,w^i)+O(1) \right)\\
		&+\sin(2\theta_i)O(1)
	\end{align*}
	as $\epsilon\to 0$.	Since each sum $\sum_i K_{h_j}(\tilde{u}^i_0,w^i_0)$ for $j=1,2$ is positive whenever all $\theta_i$ vanish (as in this case we have $\tilde{u}^i=v$ and $w^i=e^i$) and since all terms are uniformly bounded from below, it follows from continuity that there exists $\theta>0$ such that $\sum_{i} K(P_i)$ is positive whenever all $\theta_i$ are contained in the interval $[0,\theta]$ for all $\epsilon\leq \epsilon_0$.
	
	Finally, suppose there exists $\theta_i$ with $\theta_i>\theta$. We estimate the second term from below by
	\[ \sin^2(\theta)\left(\frac{1}{2\epsilon}C+O(1)  \right). \]
	This expression tends to infinity as $\epsilon\to 0$. Hence, by again using the fact that all terms are uniformly bounded from below, it follows that the sum $\sum_{i}K(P_i)$ is positive for all $\epsilon\leq\epsilon_0$ by possibly choosing $\epsilon_0$ smaller.
\end{proof}

Since the condition on $h_1'(0)-h_2'(0)$ is precisely the condition on the second fundamental form in part (i) of Theorem \ref{main} (cf.\ Lemma \ref{II}), this finishes part (i) of Theorem \ref{main} in the $C^1$-setting.

\bigskip

For the proof of part (ii) of Theorem \ref{main} we first make the following observation.
\begin{lemma}\label{Ricci_estimates}
	For the Ricci curvatures of $g$ we have the following:
	\begin{enumerate}
		\item $Ric(u_t,v_t)=\frac{1}{4\epsilon}(h_1'(0)-h_2'(0))(u,v)+O(1)$ as $\epsilon\to 0$ for any $(u,v)$ tangent to $X$,
		\item $Ric(u_t,\partial_t)=O(1)$ as $\epsilon\to 0$ for any $u$ tangent to $X$,
		\item $Ric(\partial_t,\partial_t)=\frac{1}{4\epsilon}\sum_{i=1}^{n-1}(h_1'(0)-h_2'(0))(e^i,e^i)+O(1)$ as $\epsilon\to 0$, where $(e^i)$ is any orthonormal basis in $TX$.
	\end{enumerate}
\end{lemma}

\begin{proof}
	Statement (iii) follows directly from Lemma \ref{K}. For (i) we can assume that $u=v$, since the general statement then follows from the polarisation formula. By Lemma \ref{K}, we have
	\begin{align*}
		Ric(u_t,u_t)&=\frac{1}{4\epsilon}\left( h_1'(0)-h_2'(0) \right)(u,u)+\sum_j K_t(u,e^j)-\frac{\phi_t(u,e^j)}{\psi_t(u,e^j)}+O(1)\\
		&=\frac{1}{4\epsilon}\left( h_1'(0)-h_2'(0) \right)(u,u)+\sum_j K_{h_i(0)}(u,e^j)-\frac{\phi_t(u,e^j)}{\psi_t(u,e^j)}+O(1),
	\end{align*}
	where $\{e^j\}\subset TX$ is an orthonormal basis in $TX$ with respect to $g$. By definition, the functions $\phi_t$ and $\psi_t$ only involve $g_t$ and $g_t'$. Hence, by Lemma \ref{g_t&g_t'&g_t''}, they are bounded independent of $\epsilon$ for sufficiently small $\epsilon$, so the claim follows.
	
	Finally, for (ii), we note that curvature expressions of the form $\langle R(u_t,v_t)w_t,\partial_t\rangle$ only involve $g_t$ and $g_t'$ by Lemma \ref{II} and the mixed curvature equation (see e.g.\ \cite[Theorem 3.2.5]{P}), see also \cite[Proof of Lemma 6]{BWW}. Therefore, they are bounded independent of $\epsilon$ for sufficiently small $\epsilon$ by Lemma \ref{g_t&g_t'&g_t''}.
\end{proof}

\begin{proposition}\label{C^1_scal_k}
	Suppose that $h_1$ and $h_2$ have $scal_k>0$ and that $h_1'(0)-h_2'(0)$ is $k$-positive for $1\leq k\leq n-2$ and $(k-1)$-positive for $k=n-1,n$. Then for $\epsilon$ sufficiently small $g$ has $scal_k>0$ on $X\times[-\epsilon,\epsilon]$.
\end{proposition}
\begin{proof}
	We fix a point $x\in X$ and choose an orthonormal basis $\{e^i\}$ of $T_{(x,t)}(X\times\{t\})$ for which the restriction of the Ricci tensor of $g$ to $T_{(x,t)}(X\times\{t\})$ is diagonal. Then, by Lemma \ref{Ricci_estimates}, $4\epsilon Ric$ is given in the basis $\{\partial_t,e^1,\dots,e^{n-1}\}$ by the matrix
	\[
		\begin{pmatrix}
			\sum_{i=1}^{n-1}(h_1'(0)-h_2'(0))(e^i,e^i) & 0 & \cdots & 0\\
			0 & (h_1'(0)-h_2'(0))(e^1,e^1) &  & 0\\
			\vdots &  & \ddots & \\
			0 & 0 &  & (h_1'(0)-h_2'(0))(e^{n-1},e^{n-1})
		\end{pmatrix}+O(\epsilon).
	\]
	Hence, the eigenvalues of $4\epsilon Ric$ converge to \[\left(\sum_{i=1}^{n-1}(h_1'(0)-h_2'(0))(e^i,e^i),(h_1'(0)-h_2'(0))(e^1,e^1),\dots, (h_1'(0)-h_2'(0))(e^{n-1},e^{n-1})\right)  \]
	as $\epsilon\to 0$. Thus, after possibly reordering the vectors $\{e^i\}$, the sum of the $k$ smallest eigenvalues converges either to
	\[ \sum_{i=1}^{k}(h_1'(0)-h_2'(0))(e^i,e^i) \]
	or to
	\[ \sum_{i=1}^{n-1}(h_1'(0)-h_2'(0))(e^i,e^i)+\sum_{i=1}^{k-1}(h_1'(0)-h_2'(0))(e^i,e^i). \]
	By Observation \ref{obs1} and the $k$-positive assumption, it follows directly that the first expression is positive for all $\epsilon$ sufficiently small. In the second case, if $k\leq n-2$, we can rewrite the expression as
	\[ \sum_{\substack{i=1\\i\neq k}}^{n-1}(h_1'(0)-h_2'(0))(e^i,e^i)+\sum_{i=1}^{k}(h_1'(0)-h_2'(0))(e^i,e^i), \]
	which is again positive by the $k$-positive assumption and Observation \ref{obs1}. Finally, for $k=n-1,n$ observe that both terms have at least $k-1$ summands.
	
	Since the condition of $k$-positivity is equivalent to the condition that the sum of the $k$ smallest eigenvalues is positive, see e.g.\ \cite[Lemma 1.1]{Sha1}, the claim follows.
\end{proof}

This concludes the proof of Theorem \ref{main} in the $C^1$-setting.

\bigskip\bigskip

\subsection{The $C^\infty$-smoothing}\label{C_infty}

We approach the task of smoothing the $C^1$ metric by applying mollification techniques.

\begin{lemma}\label{smoothing}
	Given that the metrics $h_1,h_2$ have $Ric_k>0$, (respectively $scal_k>0$), for some $1 \le k \le n-1$, (respectively $1 \le k \le n$), the $C^1$-interpolation metric $g$ can be smoothed in some very small $\nu$-neighbourhood of $t=\pm\epsilon$ so that the resulting $C^\infty$ metric also satisfies $Ric_k>0$, (respectively $scal_k>0$).
\end{lemma}
\begin{proof}
    We show how to smooth the metric at $t=\epsilon$, the smoothing at $t=-\epsilon$ can then be carried out analogously. To keep the notation simple, we shift the interval by $\epsilon$, to that the smoothing takes place at $t=0$.

    We choose $\nu$ sufficiently small, so that $g$ takes the form $dt^2+g_t$ on $X\times[-\nu,\nu]$. We then extend $g$ arbitrarily to all of $X \times\R$. This will be convenient for the definition of $\bar{g}$ below, but since $\bar{g}$ will only contribute to the final metric for $|t|\leq\tfrac{3\nu}{4}$, the choice of the extension will be irrelevant for the final metric.
    
    By analogy with the smoothing of \cite[Lemma 3.1]{RW1} (which deals with scaling functions for warped product metrics), for $0<\delta<\nu$ we choose a standard mollifier $\phi_\delta\colon\R\to\R$ with support in $(-\delta,\delta)$ and define the metric $\bar{g}:=dt^2+\bar{g}_t$, where
    \[ \bar{g}_t=(\phi_\delta*g_t)(t)=\int_\R\phi_\delta(t-s)g_s\,ds.  \]
    Then $\bar{g}$ is a smooth metric that $C^1$-converges uniformly on compact subsets to $g$ as $\delta\to 0$. Outside any neighbourhood of $t=0$ the convergence is $C^2$, and
    \[ \bar{g}_t''=\phi_\delta*g_t'',\]
    where we consider $g''_t$ as a metric defined almost everywhere. Similarly, since for any first or second order $X$-direction derivative $\mathcal{D}_X$ we have $\mathcal{D}_X\bar{g}_t=\phi_\delta*\mathcal{D}_Xg_t$, we see that $\mathcal{D}_X\bar{g}_t$ converges in a $C^1$-sense to $\mathcal{D}_X g_t$ as $\delta\to0$.

    For arbitrary tangent vectors $U=u+\lambda\partial_t,V=v+\mu\partial_t\in T_{(x,t)}(X\times\R)$, where $u,v\in T_x X$ and $\lambda,\mu\in\R$, we set
    \[H_2(g_t'')(U,V)=-\frac{1}{2}\left(\mu^2g_t''(u,u)-2\lambda\mu g_t''(u,v)+\lambda^2g_t''(v,v)\right).\]
    It is straightforward to see, for example using Lemmas \ref{II} and \ref{K}, that the dependence of $g(R^g(U,V)V,U)$ on $g_t''$ is precisely captured by $H_2(g_t'')(U,V)$ above. (Note that only curvatures that involve two copies of $\partial_t$ contain a term involving $g_t''$.) We can therefore write
     \[ g(R^g(U,V)V,U)=H_1(g_t,g'_t)(U,V)+H_2(g''_t)(U,V), \]
    where $H_1(g_t,g'_t)(U,V):=g(R^g(U,V)V,U)-H_2(g_t'')(U,V)$ only depends on $g_t$ and $g'_t$.
    
    Then, since $\bar{g}$ converges in a $C^1$-sense to $g$ as $\delta\to 0$, we have $H_1(\bar{g}_t,\bar{g}_t')=H_1(g_t,g_t')+R^\delta_1$, where $R^\delta_1:=H_1(\bar{g}_t,\bar{g}_t')-H_1(g_t,g_t')\to 0$ as $\delta\to 0$. Further, since $H_2(g_t'')$ depends linearly on $g_t''$, we have
    \[ H_2(\bar{g}_t'')=H_2(\phi_\delta*g_t'')=\phi_\delta*H_2(g_t''). \]
    Consequently,
    \begin{align*}
         \bar{g}(R^{\bar{g}}(U,V)V,U)&=H_1(\bar{g}_t,\bar{g}'_t)(U,V)+H_2(\bar{g}''_t)(U,V)\\
         &=H_1(g_t,g'_t)(U,V)+R^\delta_1(U,V)+(\phi_\delta*H_2(g_t''))(U,V).
    \end{align*}
    Since $\phi_\delta*H_1(g_t,g_t')\to H_1(g_t,g_t')$ as $\delta\to 0$, we have $H_1(g_t,g_t')=\phi_\delta*H_1(g_t,g_t')+R_2^\delta$, where $R_2^\delta:=H_1(g_t,g_t')-\phi_\delta*H_1(g_t,g_t')\to 0  $ as $\delta \to 0$. Hence, we have
    \begin{align*}
        H_1(g_t,g'_t)(U,V)&+(\phi_\delta*H_2(g_t''))(U,V)\\
        &=\left(\phi_\delta*H_1(g_t,g'_t)\right)(U,V)+R_2^\delta(U,V)+(\phi_\delta*H_2(g''_t))(U,V)\\
         &=\phi_\delta*\left(H_1(g_t,g'_t)(U,V)+H_2(g''_t)(U,V)\right)+R_2^\delta(U,V)\\
         &= \phi_\delta*g(R^{g}(U,V)V,U)+R_2^\delta(U,V).
    \end{align*}
    So overall we have $\bar{g}(R^{\bar{g}}(U,V)V,U)=\phi_\delta*g(R^g(U,V)V,U)+R_1^\delta(U,V)+R_2^\delta(U,V)$. Hence, any lower bound on a sum of curvatures of $g$ over $[-\nu,\nu]$, such as $Ric_k$ or $scal_k$, also holds for $\bar{g}$ over $[-\tfrac{\nu}{2},\tfrac{\nu}{2}]$ for $\delta$ sufficiently small.

    It remains to modify $\bar{g}$ so that it coincides with $g$ outside of $[-\nu,\nu]$. For that, we choose a bump function $\chi\colon\R\to[0,1]$ independent of $\delta$ that interpolates between $1$ for $|t|\geq \tfrac{3\nu}{4}$ and $0$ for $|t|\leq \tfrac{\nu}{2}$. We then set
    \[ \bar{\bar{g}}=\chi(t)g+(1-\chi(t))\bar{g}. \]
    Since $\bar{g}$ converges in a $C^2$-sense to $g$ outside any neighbourhood of $t=0$ as $\delta\to 0$, the same holds for $\bar{\bar{g}}$. Therefore, outside any neighbourhood of $t=0$, for $\delta$ sufficiently small, the metric $\bar{\bar{g}}$ satisfies any lower curvature bound that holds for $g$. Further, since $\bar{\bar{g}}$ coincides with $\bar{g}$ on $[-\tfrac{\nu}{2},\tfrac{\nu}{2}]$, the same holds on this part. Hence, for $\delta$ sufficiently small, $\bar{\bar{g}}$ is the desired smoothing.
\end{proof}

\begin{proof}[Proof of  Theorem \ref{main}]
 We have shown in Propositions \ref{C^1_result} and \ref{C^1_scal_k} that for $\epsilon>0$ sufficiently small, the $C^1$-interpolation metric $g$ has $Ric_k>0$, respectively $scal_k>0$. The theorem then follows from Lemma \ref{smoothing} above for a suitably small choice of $\nu >0$.
 \end{proof}

\begin{remark}\label{cts_dependence}
Observe that the proof of Theorem \ref{main} essentially depends on three parameters, $\epsilon$, $\nu$ and $\delta$. More precisely, there is an $\epsilon_0>0$ that depends on $k$ and the original metrics $h_1,h_2$, such that for all $\epsilon \in (0,\epsilon_0)$ the $C^1$-smoothing construction works. For the $C^\infty$ smoothing in Lemma \ref{smoothing}, further parameters $\nu>0$ and $\delta>0$ are required. For a given $\epsilon$ there is a $\nu_0=\nu_0(\epsilon,k,h_1,h_2)$ such that for any choice of $\nu \in (0,\nu_0)$ there is a $\delta_0=\delta_0(\nu,\epsilon,k,h_1,h_2)$, such that the $C^\infty$-smoothing works for all $\delta\in(0,\delta_0)$. A key point here is that the dependence of $\epsilon_0$ on $h_1,h_2$ is continuous, as is that of $\nu_0$ on $\epsilon,h_1,h_2$, and of $\delta_0$ on $\nu,\epsilon,h_1,h_2$. This will be important later in the proof of Theorem \ref{family}.
\end{remark}

\subsection{Proof of Corollary \ref{almost-nonnegative}}\label{almost-nonnegative_subsec}

First we give the precise definition of the almost-non-negative curvature conditions we consider in Corollary \ref{almost-nonnegative}.
\begin{definition}\label{D:almost_non-negative}
	\begin{enumerate}
		\item A manifold $M$ has Riemannian metrics of \emph{almost non-negative $k$-th intermediate Ricci curvature}, if for all $\epsilon>0$ there exists a Riemannian metric $g$ on $M$ such that for all orthonormal $(k+1)$-frames $(v,e^1,\dots,e^k)$ we have
		\[ \sum_{i=1}^k K_g(v,e^i) \textup{diam}(g)^2\geq -\epsilon. \]
		\item A manifold $M$ has Riemannian metrics of \emph{almost $k$-non-negative Ricci curvature}, if for all $\epsilon>0$ there exists a Riemannian metric $g$ on $M$ such that for all orthonormal $k$-frames $(e^1,\dots,e^k)$ we have
		\[ \sum_{i=1}^k Ric_g(e^i,e^i) \textup{diam}(g)^2\geq -\epsilon. \]
	\end{enumerate}
\end{definition}

\begin{proof}[Proof of Corollary \ref{almost-nonnegative}]
	We only consider case (i) as the proof in case (ii) will be entirely analogous.
	
	Recall that in the proof of Theorem \ref{main}, to obtain a $C^1$-metric, for every suitably small $\epsilon>0$ we replace  the $C^0$-metric $h_1\cup_X h_2$ on $M_1\cup_X M_2$ by a metric $g$ that agrees with $h_1$ and $h_2$ outside of the $\epsilon$-neighbourhood of $X$, and is of the form $dt^2+g_t$ otherwise. Here we identify the $\epsilon$-neighbourhood of $X$ with $X \times [-\epsilon,\epsilon] $. As $h_1\cup_X h_2$ restricted to this neighbourhood is of the form $dt^2+h_i(t)$, and since $g_t=h_i(t)+O(\epsilon),$ (with $i=1$ for $t\leq 0$ and $i=2$ for $t\geq 0$), it follows that the metric $g$ converges to $h_1\cup_X h_2$ as $\epsilon\to 0$ in a $C^0$-sense. Hence,
	\begin{equation}\label{eq:diam}
		\textup{diam}(g)\to \textup{diam}(h_1\cup_X h_2)
	\end{equation}
	as $\epsilon\to 0$. We then see from the proof of Lemma \ref{smoothing} that smoothing $g$ to a $C^\infty$ metric can be done keeping the smoothed metric arbitrarily close to $g$ in a $C^1$ sense. In particular, we can achieve this smoothing keeping the diameter arbitrarily close to that of $h_1\cup_X h_2.$
	
	Now let $\delta>0$. By assumption we have $Ric_k\geq 0>-\frac{\delta}{2d^2}$ for $h_1$ and $h_2$, with $d=\textup{diam}(h_1\cup_X h_2)$. Hence, by Corollary \ref{negative_lower_bound}, we can smooth the metric $h_1\cup h_2$ on an $\epsilon$-neighbourhood of $X$ to obtain a smooth metric $\bar{\bar{g}}$ with $Ric_k>-\frac{\delta}{2d^2}$ for all $\epsilon>0$ sufficiently small. Hence, for all orthonormal $(k+1)$-frames $\{v,e^1,\dots,e^k\}$, we have by \eqref{eq:diam},
	\[\sum_{i=1}^{k}K_{\bar{\bar{g}}}(v,e^i)\textup{diam}(\bar{\bar{g}})^2>-\frac{\delta}{2d^2}\textup{diam}(\bar{\bar{g}})^2\to -\frac{\delta}{2}>-\delta \]
	as $\epsilon\to 0$. Hence, for $\epsilon$ sufficiently small, we have
	\[\sum_{i=1}^{k}K_{\bar{\bar{g}}}(v,e^i)\textup{diam}(\bar{\bar{g}})^2>-\delta. \]
\end{proof}

 \subsection{Proof of Corollaries \ref{tot_geodesic} and \ref{double}}\label{doubling_corollaries}

	\begin{proof}[Proof of Corollary \ref{tot_geodesic}]
		By Theorem \ref{main}, we can glue $(M,g)$ to a copy of itself along the boundary and obtain a metric of $Ric_k>0$ (resp.\ $scal_k>0$). In this situation, where both manifolds are identical, it is easily checked that the metric $g_t$ defined in equation (\ref{eq:g(t)}) is symmetric about $t=0$, and it follows from this that the smooth interpolating metric constructed in the proof of Theorem \ref{main} has the same symmetry. Thus we have an isometric $\Z/2$-action on $M\cup_\partial M$ that interchanges the two copies of $M$. The fixed point set of this action is the common glued boundary $\partial M \subset M\cup_\partial M$, and this is therefore totally geodesic. Hence, by cutting $M\cup_\partial M$ along $\partial M$, we obtain a metric of $Ric_k>0$ (resp.\ $scal_k>0$) on $M$ with totally geodesic boundary.
	\end{proof}
	\begin{remark}
		Note that the metric constructed in the proof of Corollary \ref{tot_geodesic} in fact satisfies a stronger condition: it is \emph{doubling}, i.e.\ when glued to another copy of itself, the result is a smooth metric on $M\cup_{\partial} M$. 
	\end{remark}

	\begin{proof}[Proof of Corollary \ref{double}]
		Let $M^n$ be a compact manifold such that the inclusion $\partial M\hookrightarrow M$ is $d$-connected for some $1\leq d\leq n-4$. By \cite[Theorem 3]{Wa}, since $d\leq n-4$, $M$ can be obtained from $[0,1]\times\partial M$ by only attaching handles of index at least $d+1$. Turning the handle decomposition upside down, it follows that $M$ is obtained (from the empty set) by only attaching handles of index at most $n-d-1$. By \cite{Sha2}, $M$ admits a metric of positive sectional curvature such that $\II_{\partial M}$ is $(n-d)$-positive. The claim now follows from Theorem \ref{main} (ii).
	
		If $M^n$ is a closed, $d$-connected manifold, then for $M\setminus (D^n)^\circ$ the inclusion of the boundary is $d$-connected, so we can apply the first part of the theorem.
	\end{proof}

	\begin{remark}\label{bordism}
		By using the surgery results of Wolfson \cite{Wo}, we can give an alternative proof of Corollary \ref{double} as follows:
		
		First note that the manifold $M\cup_{\partial} M$ is homeomorphic to the boundary of $M\times[0,1]$. To see this, note that the boundary of $M \times [0,1]$ is
        \[M \cup_{\partial M\times\{0\}} (\partial M \times [0,1]) \cup_{\partial M\times\{1\}} M\cong M\cup_\partial M,\]
        from which the claim is clear. We can upgrade this to a diffeomorphism by rounding the corners of $M \times [0,1]$ so that $M\times [0,1]$ becomes a smooth manifold with boundary. With this in mind, we now show that the inclusion $M\cup_{\partial} M\hookrightarrow M\times[0,1]$ is $(d+1)$-connected. Indeed, the composition
		\begin{equation}\label{double_eq}
			 M \times [0,1] \to M\hookrightarrow M\cup_{\partial} M\hookrightarrow M\times[0,1],
		\end{equation}
		where the first map is projection and the second is inclusion into one of the two factors, is homotopic to the identity. Since inclusion gives an isomorphism $H_i(\partial M)\to H_i(M)$ for $i\leq d-1$, the map $H_i(M\cup_\partial M) \to H_{i-1}(\partial M)$ in the Mayer--Vietoris sequence of the push-out
		\[
			\begin{tikzcd}
				\partial M\arrow[hook]{r}\arrow[hook]{d} & M\arrow[hook]{d}\\
				M\arrow[hook]{r} & M\cup_{\partial} M
			\end{tikzcd}
		\]
		is the zero map for all $i \le d$, and it follows that the map $H_i(M)\oplus H_i(M) \to H_i(M\cup_\partial M)$ in that sequence is surjective for $i \le d$. 

It follows from the surjectivity of $H_i(\partial M)\to H_i(M)$ for $i \le d$ and the exactness of the Mayer--Vietoris sequence that the kernel of the map $H_i(M)\oplus H_i(M)\to H_i(M\cup_\partial M)$ is the diagonal. Consequently, by the definition of this last map, we now immediately deduce that the inclusion $H_i(M) \to H_i(M \cup_\partial M)$ is surjective. As both the first map in \eqref{double_eq} and the whole composition induce isomorphisms in homology, we see that the composition of the second and third maps does likewise, and hence $M \hookrightarrow M\cup_\partial M$ induces injections in homology. Thus $H_i(M) \to H_i(M \cup_\partial M)$ is an isomorphism for all $i \le d,$ and therefore the same is true of $H_i(M\cup_{\partial}M)\to H_i(M\times[0,1])$. Moreover, it is surjective for $i=d+1$ since the composition \eqref{double_eq} induces the identity in homology. It now follows from the Hurewicz theorem that provided $M$ is simply-connected, the inclusion $M\cup_{\partial}M\hookrightarrow M\times[0,1]$ is $(d+1)$-connected.

		If $M$ is non-simply-connected, we first show that the inclusion $M\cup_\partial M\hookrightarrow M\times[0,1]$ induces an isomorphism on fundamental groups. Since the composition \eqref{double_eq} is homotopic to the identity, we obtain that the map $\pi_1(M\cup_{\partial}M)\to \pi_1(M\times[0,1])$ is surjective. By van Kampen's theorem, we have an isomorphism
		\[ \pi_1(M)*_{\pi_1(\partial M)}\pi_1(M)\cong  \pi_1(M\cup_{\partial} M). \]		
		Since the induced map $\pi_1(\partial M)\to \pi_1(M)$ is surjective, it follows that the inclusion of each factor of $M$ into $M\cup_{\partial} M$ induces an isomorphism of fundamental groups. Hence, the second map in \eqref{double_eq} induces an isomorphism on fundamental groups, and since the composition of maps in \eqref{double_eq} induces an isomorphism, the map $\pi_1(M\cup_{\partial}M)\to \pi_1(M\times[0,1])$ is an isomorphism as well.

		We now consider the universal cover $\widetilde{M}\xrightarrow{\pi}M$. Then the inclusion $\partial \widetilde{M}\hookrightarrow \widetilde{M}$ is again $d$-connected, which can be seen as follows: First suppose that $\partial\widetilde{M}$ is non-connected and let $\gamma\colon[0,1]\to \widetilde{M}$ connect two connected components of $\partial\widetilde{M}$. We can assume that $\gamma(0)$ and $\gamma(1)$ get mapped to the same point under $\pi$. Then the map $\pi\circ\gamma$ defines an element of $\pi_1(M)$ with basepoint in $\partial M$. By surjectivity of the map $\pi_1(\partial M)\to\pi_1(M)$, there exists a homotopy of $\pi\circ\gamma$ fixing the basepoint, into a loop whose image lies entirely in $\partial M$. Lifting this homotopy to $\widetilde{M}$ gives a homotopy of $\gamma$ fixing the endpoints into a path with image contained in $\partial \widetilde{M}$, which is a contradiction. Hence, $\partial\widetilde{M}$ is connected and the map $\pi_0(\partial\widetilde{M})\to\pi_0(\widetilde{M})$ is a bijection.
			
		Now the map $\pi_1(\partial\widetilde{M})\to\pi_1(\widetilde{M})$ is trivially surjective, showing the claim for $d=1$. If $d>1$, we consider the following commutative diagram:
		\[
		\begin{tikzcd}
			\partial\widetilde{M}\arrow[hook]{r}\arrow{d} & \widetilde{M}\arrow{d}\\
			\partial M\arrow[hook]{r} & M.
		\end{tikzcd}
		\]
	
		Since the vertical maps are covering maps, the induced maps on $\pi_1$ are injective. Since $\pi_1(\partial M)\to\pi_1(M)$ is an isomorphism, it follows from commutativity that $\pi_1(\partial \widetilde{M})$ is trivial. Further, for $i\geq 2$, the vertical maps induce isomorphisms on $\pi_i$, showing that $\pi_i(\partial \widetilde{M})\to\pi_i(\widetilde{M})$ is an isomorphism for $i<d$ and surjective for $i=d$. Hence, the inclusion $\partial\widetilde{M}\hookrightarrow\widetilde{M}$ is $d$-connected.

		By the arguments above, the inclusion $\widetilde{M}\cup_{\partial}\widetilde{M}\hookrightarrow \widetilde{M}\times[0,1]$ is $(d+1)$-connected and since $\widetilde{M}\cup_{\partial}\widetilde{M}$ and $\widetilde{M}\times[0,1]$ are covering spaces of $M\cup_{\partial} M$ and $M\times[0,1]$, respectively, it follows from the long exact sequence of homotopy groups applied to these coverings that the inclusion $M\cup_{\partial}M\hookrightarrow M\times[0,1]$ induces isomorphism on $\pi_i$ for all $2\leq i\leq d$, and an epimorphism on $\pi_{d+1}$. Hence altogether, we obtain that the inclusion $M\cup_{\partial}M\hookrightarrow M\times[0,1]$ is $(d+1)$-connected.
		
		By deleting a disc from the interior of $M\times[0,1]$, we obtain a bordism between $M\cup_{\partial} M$ and $S^n$. As in the proof of Corollary \ref{double}, we conclude that $(M\times[0,1])\setminus (D^n)^\circ$ is obtained from $S^n\times[0,1]$ by only attaching handles of index at most $n-d-2$. By considering the effect on the boundary, it follows that $M\cup_{\partial} M$ is obtained from $S^n$ by surgeries of codimension at least $d+2$. Since the condition $scal_k>0$ can be preserved under surgeries of codimension at least $n-k+2$ by \cite{Wo} (provided $k<n$), the claim follows.
	\end{remark}

\bigskip

\section{Proof of Theorem \ref{obs}}\label{obs_proof}

As noted in the Introduction, the proof of Theorem \ref{obs} follows that of the main theorem in \cite{BWW} very closely. This latter theorem is the special case of Theorem \ref{obs} corresponding to $k=n-1,$ i.e. the case of positive Ricci curvature. Besides the stronger curvature conditions considered in Theorem \ref{obs}, there is a further point of difference: here we compute explicit parameter ranges for which the result holds. This was not done in \cite{BWW}. For the convenience of the reader we will now outline the key features of the argument in \cite{BWW} that are relevant for our purposes.

We begin by considering the observer moduli space of Riemannian metrics on $S^n$, ${\mathcal M}_{x_0}(S^n).$ This is the quotient of the space of all Riemannian metrics ${\mathcal R}(S^n)$, without any curvature restrictions, by the free action of the observer diffeomorphism group $\text{Diff}_{x_0}(S^n).$ As ${\mathcal R}(S^n)$ is a contractible space, this means that ${\mathcal M}_{x_0}(S^n)$ can be identified with the classifying space of the observer diffeomorphism group, $\text{BDiff}_{x_0}(S^n)$. In \cite{FH}, Farrell and Hsiang in effect compute the rational homotopy groups $\pi_j\bigl(\text{BDiff}_{x_0}(S^n)\bigr)\otimes \Q$ for $n$ odd and sufficiently large compared to $j$. The optimal range for $n$ and $j$ was recently identified in \cite{KRW}, but see also \cite{E} for a synopsis of these developments. The resulting theorem is:
\begin{theorem}\label{Farrell}
Suppose that $n$ is odd, $n \ge 11$, and $j \le n-4.$ Then $$\pi_j\bigl(\text{BDiff}_{x_0}(S^n)\bigr)\otimes \Q \cong
\begin{cases}
\Q \, \text{ if }j=4\ell\text{ for some }\ell\in\N, \\
0 \, \text{ otherwise.} \\
\end{cases}
$$
\end{theorem}
Thus we automatically know $\pi_j\bigl({\mathcal M}_{x_0}(S^n)\bigr) \otimes \Q$ in the same range.

As the space $\text{BDiff}_{x_0}(S^n)$ classifies bundles with fibre $S^n$ and structure group $\text{Diff}_{x_0}(S^n)$, any continuous map $f\colon S^{4\ell} \to \text{BDiff}_{x_0}(S^n)$ gives rise to a bundle of this form with base space $S^{4\ell}$. Moreover, by identifying $\text{BDiff}_{x_0}(S^n)$ with ${\mathcal M}_{x_0}(S^n)$, it can be shown that such a bundle comes equipped with a natural {\it fibrewise Riemannian metric} determined by $f$. See \cite[page 61]{TW} for a detailed description of how the bundle and metric arise simultaneously. Note that a fibrewise Riemannian metric on a fibre bundle $F \hookrightarrow E \to B$ is a choice of Riemannian metric $g_b$ on each fibre $F_b$ which varies smoothly with respect to $b \in B$, in other words, a smoothly varying choice of inner product on the vertical tangent bundle. We say that a fibrewise metric has $Ric_k>0$ if each Riemannian manifold $(F_b,g_b)$ has $Ric_k>0.$ 

The relevance of fibrewise metrics for Theorem \ref{obs} is this: if a bundle which represents a non-zero element of $\pi_{4\ell}\bigl({\mathcal M}_{x_0}(S^n)\bigr)$ can be equipped with a fibrewise $Ric_k>0$ metric, this means that the element can be `lifted' to $\pi_{4\ell}\bigl({\mathcal M}^{Ric_k>0}_{x_0}(S^n)\bigr)$, i.e. it lies in the image of the natural inclusion map $\pi_{4\ell}\bigl({\mathcal M}^{Ric_k>0}_{x_0}(S^n)\bigr) \to \pi_{4\ell}\bigl({\mathcal M}_{x_0}(S^n)\bigr)$, showing that $\pi_{4\ell}\bigl({\mathcal M}^{Ric_k>0}_{x_0}(S^n)\bigr)$ is non-trivial. For elements in rational homotopy, this observation was exploited in \cite{BWW} in the context of positive Ricci curvature, having been used previously in \cite{BHSW} for positive scalar curvature. See for example \cite[section 7.2]{TW} for a detailed summary.
 
It can be seen from work of Goette, \cite[section 5]{G} and Igusa, \cite[sections 6.4, 6.5]{I} that for $n,j$ as in Theorem \ref{Farrell}, non-zero elements of $\pi_{4\ell}\bigl(\text{BDiff}_{x_0}(S^n)\bigr) \otimes \Q=\pi_{4k}\bigl({\mathcal M}_{x_0}(S^n)\bigr) \otimes \Q$ can be represented by {\it Hatcher bundles}. These are smooth non-linear bundles which are homeomorphic but not diffeomorphic to the trivial bundle $S^n \times S^{4\ell} \to S^{4\ell}.$ Hatcher bundles admit a reasonably explicit construction, and in particular they can be decomposed into a pair of identical disc bundles. The proof of the main theorem in \cite{BWW} shows that Hatcher bundles can be equipped with fibrewise Ricci positive metrics, and this essentially boils down to establishing two things: that the `Hatcher disc bundles' can be equipped with fibrewise Ricci positive metrics, and that a pair of identical Hatcher disc bundles can be glued fibrewise preserving the fibrewise Ricci positive condition. This latter problem can be addressed by a Ricci positive family gluing result (\cite[Theorem 3]{BWW}). In order to utilize this and ultimately prove the theorem, it suffices to show that Hatcher disc bundles can be equipped with a fibrewise Ricci positive metric for which the boundary of each fibre is convex. 

The proof of Theorem \ref{obs} is completely analgous to the above. As in \cite{BWW}, the first step is to establish a family gluing result, which here means a family version of Theorem \ref{main}. This is as follows:

\begin{theorem}\label{family}
Let $\pi_i\colon E_i \to B$, $i=1,2$ be smooth compact fibre bundles with
fibre $M_i,$ where $\partial M_i \neq \emptyset.$ Suppose that each of
these bundles is equipped with a smoothly varying family of fibrewise
$Ric_k>0$ metrics $\{h_i(b)\}_{b \in B}.$ Assume that with respect
to these metrics, there is an isometry $\phi_b:  \partial \pi_1^{-1}(b) \to \partial \pi_2^{-1}(b)$ for each $b \in B$. Moreover assume that this family of maps varies smoothly with respect to $b$, in the sense that given any local trivializations $\theta_i:E_i|_{\pi_i^{-1}(U)} \cong U \times M_i$ over $U \subset B$, the maps $$\bar{\phi}_b:=\theta_2|_{\pi_2^{-1}(b)}\circ \phi_b \circ \theta_1^{-1}|_{\{b\} \times \partial M_1}:\partial M_1 \to \partial M_2,$$ (where we have identified $\{b\} \times \partial M_i$ with $\partial M_i$ on the right,) vary smoothly with respect to $b \in U$.
Finally, assume that $\II_{\partial \pi^{-1}_1(b)}+\II_{\partial \pi^{-1}_2(b)}$ is positive definite for each $b\in B$, where $\II_{\partial \pi^{-1}_i(b)}$ denotes the second fundamental form of the boundary $\partial \pi^{-1}_i(b) \subset (\pi_i^{-1}(b),h_i(b)).$ Then the fibrewise $C^{0}$-metric $h:=\{h_1(b)\cup_{\phi_b} h_2(b)\}_{b \in B}$ on $E_1\cup_{\{\phi_b\}_{b \in B}}E_2$ can be smoothed within fibrewise $Ric_k>0$ in such a way that the resulting metric agrees with the original outside a neighbourhood of the glued boundaries.
\end{theorem}

\begin{proof}
The proof is analogous to that for \cite[Theorem 10]{BWW}. In Remark \ref{cts_dependence} we noted that the proof of Theorem \ref{main} depends on positive parameters, $\epsilon$, $\nu$ and $\delta$, for which the upper bounds depend continuously on the metrics $h_1,h_2$ in the case of $\epsilon$, on $h_1,h_2,\epsilon$ in the case of $\nu$, and on $h_1,h_2,\epsilon,\nu$ in the case of $\delta$. As a consequence, uniform choices for these three parameters can be made for all fibres since the bundles are compact. With a suitable global choice of $\epsilon,\nu,\delta$, we apply Theorem \ref{main} to each of the bundle fibres individually. The result will now follow provided we can show that the resulting collection of smooth fibre metrics vary smoothly with respect to $b$. However the smooth dependence of this construction on the metrics  $h_i(b)$ is evident from equation (\ref{eq:g(t)}) in the case of the $C^1$-smoothing, and from the explicit construction of the proof of Lemma \ref{smoothing} in the $C^\infty$ case.
\end{proof}

\begin{proof}[Proof of Theorem \ref{obs}]
In \cite[section 4]{BWW}, it is shown that a Hatcher disc bundle can be equipped with a fibrewise Ricci positive metric, such that each fibre is isometric to a smooth `ellipsoid' ${\mathcal E}^n$ with convex boundary. Here ${\mathcal E}^n \subset D^{p+1} \times D^q$, where $n=p+q+1$, and the product of discs is equipped with a product metric in which each factor is rotationally symmetric with positive sectional curvature. This product is, of course, a manifold with corners, and an explicit construction is provided for creating ${\mathcal E}^n$ by `shaving off' these corners in a manner which creates the desired smooth convex boundary, see \cite[section 4.3]{BWW}. The induced metric on ${\mathcal E}^n$ is just the restriction of the product metric, which clearly has positive Ricci curvature. Moreover, by \cite[Theorem A]{RW2}, this product metric has $Ric_k>0$ for $k \ge \max\{p+2,q+1\}.$ In other words, the metric construction in \cite{BWW} actually yields Hatcher disc bundles with convex fibre boundaries and fibrewise $Ric_k>0$ for $k$ in this range. Note that, in order to have $k\leq n-1$, (so that $Ric_k$ is defined), we need in particular to assume $p+2\leq n-1$, i.e.\ $q\geq 2$.

The next task is to show that a pair of identical Hatcher disc bundles with fibrewise $Ric_k>0$ metrics and convex fibre boundaries can be glued in the obvious manner preserving the fibrewise $Ric_k>0$ condition. But this is precisely the content of Theorem \ref{family}, so the construction of Hatcher (sphere) bundles with fibrewise $Ric_k>0$ metrics is complete.

To conclude the proof, we establish details (not considered in \cite{BWW}) of the parameter ranges for which the Theorem holds. Below, all page numbers refer to \cite{BWW}.

Given $\ell \in \N,$ we first consider what it means for $n=p+q+1$ to be `sufficiently large' compared to $\ell$ in the context of the constructions in \cite[\S3]{BWW}. The requirements on $p$ in relation to $\ell$ are twofold, and are specified on page 3026: both $\pi_{4\ell-1}(\text{O}(p))$ and $\pi_{4\ell-1+p}(S^p)$ must be independent of $p$. By the Freudenthal suspension theorem, the latter homotopy group will be independent of $p$ provided $p\ge 4\ell+1.$ The former group will be independent of $p$ if $p \ge 4\ell+1$, as a consequence of the long exact sequence in homotopy corresponding to the fibration $\text{O}(p) \to \text{O}(p+1) \to S^p.$ Thus overall we need $p \ge 4\ell+1.$ (So $p \ge 5$.)

We must also consider the corresponding conditions arising from Theorem \ref{Farrell}. This yields the requirement $n \ge 4\ell +4$, but, since $q\geq 2$, this is redundant as it is implied by the condition $p \ge 4\ell+1$ identified above.

There is one more dimension restriction specified in \cite[\S3]{BWW}: as explained on page 3027, we need $q \ge p+1$. However there are also two further conditions which do not appear explicitly in \cite{BWW}. In the proof, but not the statement of \cite[Lemma 4]{BWW}, there is a requirement that $p \ge 4\ell.$ (This lemma reproduces \cite[Proposition 5.4]{G}, where the result is actually proved.) Clearly, in our situation this inequality is automatically satisfied. There is a further condition which arises in the computation of analytic torision in \cite{G}, specifically in the proof of Theorem 5.13. This is needed to establish that the Hatcher bundles we construct are smoothly non-trivial bundles. The condition here is $q\ge p+2$. (Since $p \ge 5,$ this means that $q \ge 7.$) 

Overall then, we need $p \ge 4\ell+1$ and $q \ge p+2$, with of course $n=p+q+1$. For given $\ell$, the minimum possible value of $n$ is therefore $4\ell+1+4\ell+3+1=8\ell+5,$ and in particular $n \ge 13$. (Compare \cite[Theorem 6.4.2]{I}.)

Turning our attention to curvature, we noted above that a product of positive sectional curvature metrics on $D^n=D^{p+1} \times D^q$ has $Ric_k>0$ for $k \ge \max\{p+2,q+1\}.$ Given that in our situation $q \ge p+2$, we therefore have $k \ge q+1$. Thus the minimum possible value of $k$ (corresponding to the strongest intermediate Ricci curvature condition) will occur when $q$ is minimal. Setting $p=(n-3)/2$ and $q=(n+1)/2$ gives $q=p+2$ as well as $p+q+1=n$, and it is easy to see that these admissible values of $p$ and $q$ yield the smallest possible $q$. Thus $k \ge (n+3)/2.$

In summary, for any given $\ell \in \N$, and any $n \ge 8\ell+5$, there exist $p,q$ satisfying the requirements of the construction in \cite{BWW} which result in a fibrewise $Ric_k>0$ metric on an $S^n$-Hatcher bundle over $S^{4\ell}$ for any $k \ge (n+3)/2.$ Hence there is a non-trivial element in the rational homotopy group of the corresponding observer moduli space of metrics in dimension $4\ell,$ as claimed by the Theorem.
\end{proof}

\bigskip

\noindent\textbf{Data availability} Data sharing not applicable to this article as no datasets are analysed or generated in this work.

\section*{Declarations}

On behalf of all authors, the corresponding author states that there is no conflict of interest.

%%%%%%%%%%

\bigskip\bigskip

\noindent{\it Philipp Reiser,
	Email: philipp.reiser.math@mailbox.org.\\
	Corresponding author }\\

\bigskip

\noindent {\it David Wraith, 
	Department of Mathematics and Statistics, 
	National University of Ireland Maynooth, Maynooth, 
	County Kildare, 
	Ireland. 
	Email: david.wraith@mu.ie.}

\end{document}